\documentclass{amsart}

\usepackage{amsmath}
\usepackage{amssymb}
\usepackage{amscd}
\usepackage{hyperref}

\newtheorem{theorem}{Theorem}
\newtheorem{thm}[theorem]{Theorem}
\newtheorem{cor}[theorem]{Corollary}
\newtheorem{lem}[theorem]{Lemma}

\theoremstyle{definition}
\newtheorem{defn}[theorem]{Definition}

\newtheorem{ques}[theorem]{Question}

\newtheorem{rem}[theorem]{Remark}

\newtheorem{ex}[theorem]{Example}

\theoremstyle{remark}

\newcommand{\mbb}{\mathbb}
\newcommand{\QQ}{\mbb{Q}}

\newcommand{\PP}{\mbb{P}}

\newcommand{\mc}{\mathcal}

\newcommand{\mcF}{\mc{F}}

\newcommand{\OO}{\mc{O}}

\newsavebox{\sembox}
\newlength{\semwidth}
\newlength{\boxwidth}

\newsavebox{\semrbox}
\newlength{\semrwidth}
\newlength{\boxrwidth}

\title
{Separable rational connectedness and stability}

\author[Tian]{Zhiyu Tian}
\address{
Department of Mathematics 253-37\\
California Institute of Technology \\
Pasadena, CA, 91125}
\email{tian@caltech.edu}

\date{\today}

\begin{document}


\begin{abstract}
In this short note we prove that in many cases the failure of a variety to be separably rationally connected is caused by the instability of the tangent sheaf (if there are no other obvious reasons). A simple application of the results proves that a smooth Fano complete intersection is separably rationally connected if and only if it is separably uniruled. In particular, a general such Fano complete intersection is separably rationally connected.
\end{abstract}


\maketitle


It has now become clear that the geometry of varieties are in a large part controlled by rational curves. And it is desirable to single out the class of varieties which contains lots of rational curves. In characteristic zero, such class of varieties is rationally connected. 
\begin{defn}
A variety $X$ is rationally connected if there is a family of rational curves
\[
u: \PP^1 \times U \to X,
\]
such that the two point evaluation morphism
\[
u^{(2)}: \PP^1 \times \PP^1 \times U \to X \times X
\]
is dominant.
\end{defn}
Over an uncountable algebraically closed field, this condition is the same as the geometric condition that there is a rational curve through two general points in $X$. Rationally connected varieties in char $0$ have been identified as the correct generalization of rational surfaces to all dimensions \cite{Kollar96}.

However, in positive characteristic, rationally connected varieties are not the correct generalization of rational surfaces since there are inseparable unirational parameterizations of varieties of general type. Instead, one should look at separably rationally connected varieties.
\begin{defn}
A variety $X$ is separably rationally connected if the two point evaluation morphism
\[
u^{(2)}: \PP^1 \times \PP^1 \times U \to X \times X
\]
 is dominant and separable.
\end{defn}
 But then an interesting question arises:
\begin{ques}\label{1}
What makes a variety (not) separably rationally connected?
\end{ques}
In particular, the following is a well-known open question:
\begin{ques}\label{2}
Is every smooth Fano hypersurface separably rationally connected?
\end{ques}

In this short note we try to suggest an answer to Question \ref{1} in some cases, and relate it to Question \ref{2}. The basic observation is the following.
\begin{thm}\label{thm:PicOne}
Let $X$ be a normal projective variety of Picard number one over an algebraically closed field of positive characteristic. Assume that the smooth locus $X^\text{sm}$ of $X$ is separably uniruled. If $X$ is not separably rationally connected, then the tangent sheaf $T_X$ is unstable.
\end{thm}
For the ease of the reader, we recall the definition of (separable) uniruledness.
\begin{defn}
A variety $X$ is (separably) uniruled if there is a family of rational curves
\[
u: \PP^1 \times U \to X,
\]
such that the morphism $u$ 
is dominant (and separable), and non-constant along the $\PP^1$ factor.
\end{defn}

\begin{rem}
If $X$ is $\QQ$-factorial, then one can prove that the smooth locus $X^\text{sm}$ is rationally connected and two general points can be connected by a free rational curve. That is, $X$ is freely rationally connected (FRC) as in \cite{ShenFRC}, Definition 1.2.
\end{rem}
\begin{proof}[Proof of Theorem \ref{thm:PicOne}]
Since the smooth locus $X^{\text{sm}}$ is separably uniruled, there is a free curve, i.e. a morphism $f:\PP^1 \to X^\text{sm}$ such that 
\[
f^*T_X\cong \OO(a_1)\oplus \ldots \oplus \OO(a_r) \oplus \OO(a_{r+1}) \oplus \OO(a_n), 
\]
\[
a_1 \geq a_2 \geq \ldots a_r>a_{r+1}=\ldots=a_n=0, n=\dim X.
\]

Define the positive rank $R$ of $X$ to be the maximum of such numbers $r$. A free curve is called \emph{maximally free} if the pull-back of the tangent bundle has $R$ positive summands.

Given a general point $x \in X$, by \cite{ShenFRC}, Proposition 2.2, there is a well-defined subspace $D(x) \subset T_X|_x$, as the subspace of the positive directions of a maximally free curve at $x$ (i.e. $D(x)$ is independent of the choice of the maximally free curve). Furthermore,  over an open subset $U$ of $X$, the subspaces of $D(x)$ glues together to a (locally free) coherent subsheaf of $T_X$ (loc. cit. Proposition 2.5). Denote by $\mathcal{D}$ the saturated subsheaf of $T_X$ which extends the locally free subsheaf given by $D(x), x \in U$. Obviously the rank of $\mathcal{D}$ is $R$.

Let $\phi: \PP^1 \to X^{\text{sm}}$ be a maximally free curve. Then we have
\[
\phi^* {T_X} \cong \OO(a_1)\oplus \ldots \oplus \OO(a_R) \oplus \OO(a_{R+1}) \oplus\ldots \oplus \OO(a_n), 
\]
\[
\phi^* \mathcal{D} \cong \OO(a_1)\oplus \ldots \oplus \OO(a_R),
a_1 \geq a_2 \geq \ldots a_R>a_{R+1}=\ldots=a_n=0,
\]
(c.f. the paragraph after Corollary 3.2, loc. cit.)

Thus we have the equality between the first Chern classes $c_1(\mathcal{D})=c_1(T_X)$. Here we use the fact that $X$ is Picard number one and we can prove the equality by taking intersection numbers with a maximally free curve.

So if $R<n$, or equivalently, $X^{\text{sm}}$ is not separably rationally connected, then the tangent sheaf is unstable in the sense of Mumford. Indeed, we have 
\[
\frac{c_1(\mathcal{D}) \cdot H^{n-1}}{\text{rank} \mathcal{D}}>\frac{c_1(T_X)\cdot H^{n-1}}{n}>0
\]
for some ample divisor $H$.
\end{proof}

\begin{rem}
Koll\'ar constructed examples of degree $p$ branched covers of $\PP^n$ in characteristic $p$ which are separably uniruled, rationally connected, but not separably rationally connected (Exercise 5.19, Chap. V, \cite{Kollar96}). In his examples, the sheaf of differentials $\Omega_X$ has an unexpected quotient sheaf, whose dual is basically the sheaf $\mathcal{D}$ in the proof.
\end{rem}

One application of Theorem \ref{thm:PicOne} is the following. 
\begin{cor}
Consider the following three properties of a smooth Fano complete intersection $X$ of dimension at least $3$:
\begin{enumerate}
\item
$X$ is separably uniruled.
\item $X$ is separably rationally connected.
\item $X$ is rationally connected.
\end{enumerate}
Then the first two properties are equivalent to each other and imply the last one. 
\end{cor}
\begin{proof}
If $X$ is a linear subspace or a smooth quadric hypersurface, the statement is trivial.

In the following we assume $X$ is neither a linear subspace nor a quadric hypersurface. Then we use the following to show that the sheaf of differentials $\Omega_X$ is stable for such complete intersections. A proof of the stability over a field of characteristic $0$ can be found in \cite{StabilityTangent}. The proof only uses the fact that $H^0(X, \Omega_X^q(q-1))=0$ and the fact that $X$ has Picard number one. All these facts remain true in positive characteristic. The general vanishing result is proved in Lemma 3.3, \cite{Olivier}. But below we include a proof for completeness, which is essentially the same as the one in \cite{Olivier}.

The vanishing result we will need is the following.
\begin{lem}[\cite{Olivier}]\label{lem:vanishing}
Let $X \subset \PP^n$ be a smooth complete intersection of degree $(d_1, \ldots, d_c)$. Then we have $H^p(X, \Omega_X^q(t))=0$, for all $p+q < \dim X=n-c, t<q-p$.
\end{lem}

Assuming this lemma, the proof proceeds as in \cite{StabilityTangent}. For any subsheaf $\mcF \subset \Omega_X$, we may assume $\mcF$ is reflexive of rank $r < \dim X$. Thus $\det \mcF$ is an invertible subsheaf of $\Omega_X^r$. Since $X$ has Picard number one by the Grothendieck-Lefschetz Theorem (Corollary 3.2, Chap. IV, \cite{AmpSubVar}), we know $\det \mcF$ is isomorphic to $\OO_X(k)$ for some $k$. Then $H^0(X, \Omega_X^r(-k))=0$. So by Lemma \ref{lem:vanishing}, $-k\geq r$. Then
\[
\mu(\mcF)=\frac{\det \mcF \cdot \OO(1)^{\dim X-1}}{r}\leq -\deg X < \mu(\Omega_X)=\deg X \cdot \frac{\sum_1^c d_i-n-1}{\dim X}.
\]
Here in the last inequality we use the assumption that $\sum d_i-1-c>0$ (note that the proof of Corollary 0.3 in \cite{StabilityTangent} wrongly assumes this without the restriction on the complete intersection not being a linear subspace or a hyperquadric).
\end{proof}

Now we use inductions to prove Lemma \ref{lem:vanishing}, which is divided into the following two lemmas.

\begin{lem}
Let $X\cong \PP^n$. Then $H^p(X, \Omega_X^q(t))=0$, for all $p+q < \dim X=n, t<q-p$.
\end{lem}
\begin{proof}
We use induction on $q$. The statement is true if $q=0$. We have the Euler sequence
\[
0 \to  \Omega_X \to V \otimes \OO(-1) \to \OO \to 0.
\]
Taking exterior powers gives
\[
0 \to \Omega_X^q \to \wedge^q V \otimes \OO(-q) \to \Omega_X^{q-1} \to 0,
\]
which gives the following exact sequence of cohomology groups:
\[
H^{p-1}(X, \Omega_X^{q-1}(t)) \to H^p(X, \Omega_X^q(t)) \to H^p(X, \wedge^q V \otimes \OO(-q)).
\]
By assumption on $p, q$, $H^p(X, \wedge^q V \otimes \OO(t-q))=0$. And $H^{p-1}(X, \Omega_X^{q-1}(t))=0$ by the induction hypothesis since $(p-1)+(q-1) < \dim X, t < (q-1)-(p-1)$. The statement follows.
\end{proof}
\begin{lem}
Let $Y \subset (X, \OO_X(1))$ be a smooth hypersurface of degree $d \geq 2$. And assume that $H^p(X, \Omega_X^q(t))=0$, for all $p+q < \dim X, t<q-p$. Then $H^p(Y, \Omega_Y^q(t))=0$, for all $p+q < \dim Y, t<q-p$.
\end{lem}
\begin{proof}
We use induction on $q$. The $q=0$ case is easy.

We have short exact sequences:
\[
0 \to \OO_Y(-d) \to \Omega_X|_Y \to \Omega_Y \to 0,
\]
\[
0 \to \Omega_X^q(-d) \to \Omega_X^q \to \Omega_X^q|_Y \to 0,
\]
\[
0 \to \Omega^{q-1}_Y(-d) \to \Omega_X^q|_Y \to \Omega_Y^q \to 0.
\]
So we have exact sequence of cohomology groups:
\[
H^p(Y, \Omega_X^q(t)|_Y) \to H^p(Y, \Omega_Y^q(t)) \to H^{p+1}(Y, \Omega_Y^{q-1}(t-d)),
\]
\[
H^p(X, \Omega_X^q(t)) \to H^p(Y, \Omega_X^q(t)|_Y) \to H^{p+1}(X, \Omega_X^q(t-d)).
\]
The assumptions on $p, q, t, d$ implies that 
\[
(p+1)+q < \dim X, 
\]
\[
t-d \leq t-2<q-p-2=(q-1)-(p+1)<q-(p+1).
\]
So 
\[
H^p(Y, \Omega_X^q(t)|_Y)=H^{p+1}(Y, \Omega_Y^{q-1}(t-d))=H^p(X, \Omega_X^q(t))= H^{p+1}(X, \Omega_X^q(t-d))=0,
\]
by the induction hypothesis and the assumption on the cohomology groups of $X$.
 Note the $d\geq 2$ assumption is crucial to get the vanishing of $H^{p+1}(Y, \Omega_Y^{q-1}(t-d))$.
\end{proof}
As a further corollary, we give a different proof of the following result of Chen-Zhu.
\begin{cor}[\cite{ChenZhu}]
A general Fano complete intersection of dimension at least $3$ is separably rationally connected. 
\end{cor}
\begin{proof}[Sketch of proof]
It suffices to prove separable uniruledness of a general such complete intersection, which is Exercise 4.4, Chap. V, \cite{Kollar96} for hypersurface, and Proposition 2.13 in Debarre's book \cite{DebarreBook} for complete intersections of index at least $2$. The remaining case can be proved in the same way, i.e. by writing down an explicit complete intersection which contains a free curve. 

The key point that make separable uniruledness much easier to prove than the separable rational connectedness case is that we only need to work with lines and conics for separable uniruledness, while for separable rational connectedness, the degree of rational curves grows like the dimension of the variety.
\end{proof}

In the higher Picard number case, the tangent sheaf may fail to be semi-stable for many reasons, for example, if the variety has a fibration structure. From a more positive perspective, we would like to prove that if the tangent sheaf is semistable and if there are no other obvious reasons for the variety to be not separably rationally connected, then the variety is separably rationally connected.

However, it is not clear what should the term ``obvious reasons" mean. Below we suggest one possibility.

It is very easy to show that on a smooth projective separably rationally connected variety $X$, the group of rational one cycles modulo numerical equivalence $N_1(X)_\QQ$ is generated by (very) free rational curves. Thus if this group is not generated by free rational curves, then the variety $X$ is obviously not separably rationally connected. Then essentially the same proof as in \ref{thm:PicOne} gives the following.
\begin{thm}\label{thm:general1}
Let $X$ be a smooth Fano variety over an algebraically closed field of positive characteristic. Assume that $X$ is separably uniruled and the group of rational one cycles modulo numerical equivalence $N_1(X)_\QQ$ is generated by free rational curves. If the tangent sheaf of $X$ is semi-stable, then $X$ is separably rationally connected.
\end{thm}
Note that the classes of maximally free rational curves span $N_1(X)_\QQ$ by a simple deformation argument. Similar as in the Picard number one case, one can then conclude the equality of the first Chern class of $\mathcal{D}$ and $T_X$, at least numerically, by evaluating them on the maximally free curves. The Fano condition is also important to get the slope inequality in the desired form.

\begin{rem}
The conditions imply that $X$ is FRC by the quotient construction (Theorem 4.13, Chap. IV,\cite{Kollar96}, see also the proof of Corollary 4.14). But they are too restrictive. For example it suffices to assume that there is a contraction of $X$ which contracts all the divisors which do not intersect the free curves. 
\end{rem}

\begin{ex}
There are rationally connected (even FRC), separably uniruled smooth projective varieties whose group of rational one cycles modulo numerical equivalence $N_1(X)_\QQ$ is not generated by free rational curves. Indeed, take one of Koll\'ar's examples (Exercise 5.19, Chap. V, \cite{Kollar96}). And let $X$ be a resolution of singularities, which exists by the local description of the singularities. Finally let $Y$ be the blow-up of $X$ along a smooth point. Then the intersection number of the exceptional divisor $E$ with any free rational curve on $Y$ is $0$. Otherwise one can construct very free rational curves on $X$ from a free rational curve on $Y$ which has positive intersection number with the exceptional divisor.
\end{ex}

\textbf{Acknowledgment:} The idea of the paper comes from a lecture on foliations in the summer school ``rational points, rational curves, and entire holomorphic curves on projective varieties". I would like to thank the organizers for their hard work and the all the lecturers in the summer school for their enlightening lectures. Finally, I would like to thank Prof. Olivier Debarre for suggesting the reference \cite{Olivier}.

This paper is dedicated to my dearest friend, Neipu, for his accompany in the time of happiness and in the time of sorrow, and for his strong belief in wait and hope. May he rest in peace.
\bibliographystyle{alpha}
\bibliography{MyBib}

\end{document}